\documentclass[11pt,reqno]{amsart}
\usepackage{amsmath,amsfonts,amsthm}
\usepackage[colorlinks]{hyperref}
\usepackage[all]{xy}

\title{On two rationality conjectures for cubic fourfolds}
\author{Nicolas Addington}
\email{adding@math.duke.edu}
\address{\noindent Department of Mathematics \\
Duke University, Box 90320 \\
Durham, NC 27708-0320 \\
United States}

\newcommand \A {\mathcal A}

\renewcommand \O {\mathcal O}
\renewcommand \P {\mathbb P}
\newcommand \Q {\mathbb Q}
\newcommand \Z {\mathbb Z}
\DeclareMathOperator \Hilb {Hilb}
\DeclareMathOperator \Ext {Ext}
\DeclareMathOperator \disc {disc}
\DeclareMathOperator \Gr {Gr}
\DeclareMathOperator \ch {ch}
\DeclareMathOperator \td {td}

\DeclareMathOperator \Coh {Coh}

\newcommand \Ktop {K_\mathrm{top}}
\newcommand \Knum {K_\mathrm{num}}
\newcommand \alg {\mathrm{alg}}
\newcommand \prim {\mathrm{prim}}

\newcommand \Kthreen {K3$^{[n]}$}

\newtheorem{thm}{Theorem}
\newtheorem{prop}[thm]{Proposition}
\newtheorem{cor}[thm]{Corollary}
\newtheorem{lem}[thm]{Lemma}

\renewcommand \phi \varphi

\begin{document}

\begin{abstract}
Motivated by the question of rationality of cubic fourfolds, we show that a cubic $X$ has an associated K3 surface in the sense of Hassett if and only if the variety $F$ of lines on $X$ is birational to a moduli space of sheaves on a K3 surface, but that having $F$ birational to $\Hilb^2(\text{K3})$ is more restrictive. We compare the loci in the moduli space of cubics where each condition is satisfied.
\end{abstract}

\maketitle

It is widely expected that a smooth complex cubic fourfold $X$ is rational if and only if it has an associated K3 surface in the sense of Hassett \cite{hassett} or Kuznetsov \cite{kuznetsov}.  New work of Galkin and Shinder \cite{gs} suggests instead that if $X$ is rational then the variety $F$ of lines on $X$ is birational to the Hilbert scheme of two points on a K3 surface.  The purpose of this note is to clarify the relationship between these two conditions.  The latter is somewhat stronger.

First let us recall Hassett's Noether--Lefschetz divisors $\mathcal C_d$ in the moduli space $\mathcal C$ of cubic fourfolds \cite[\S3.2]{hassett}.  For a very general cubic $X$, the algebraic lattice $H^{2,2}(X,\Z) := H^{2,2}(X) \cap H^4(X,\Z)$ is generated by $h^2$, the square of the hyperplane class.  A \emph{special cubic of discriminant $d$} is one for which there is a primitive sublattice $K \subset H^{2,2}(X,\Z)$ of rank 2 and discriminant $d$ that contains $h^2$.  Such cubics form an irreducible divisor $\mathcal C_d \subset \mathcal C$, non-empty if and only if
\begin{equation} \label{condition*} \tag{$*$}
d > 6 \mathrm{\ and\ } d \equiv 0 \mathrm{\ or\ } 2 \pmod 6.
\end{equation}
Moreover there exists a polarized K3 surface $S$ such that $K^\perp \subset H^4(X,\Z)$ is Hodge-isometric to $H^2_\prim(S,\Z)(-1)$ if and only $d$ satisfies the further condition
\begin{equation} \label{condition**} \tag{$**$}
\text{$d$ is not divisible by 4, 9, or any odd prime $p \equiv 2 \pmod 3$.}
\end{equation}
Using the Eisenstein integers one can show that \eqref{condition**} is equivalent to saying that $d$ is the norm of a primitive vector in the lattice $A_2 = \left( \begin{smallmatrix} 2 & -1 \\ -1 & 2 \end{smallmatrix} \right)$, or that $d$ divides $2n^2+2n+2$ for some integer $n$.

\begin{thm} \label{theorem**}
The following are equivalent:
\begin{enumerate}
\item $X \in \mathcal C_d$ for some $d$ satisfying \eqref{condition**}.
\item The transcendental lattice $T_X \subset H^4(X,\Z)$ is Hodge-isometric to $T_S(-1)$ for some K3 surface $S$.
\item $F$ is birational to a moduli space of stable sheaves on $S$.
\end{enumerate}
\end{thm}
\noindent By a recent result of Bayer and Macr\`i \cite[Thm.~1.2(c)]{bm}, this last condition is equivalent to saying that $F$ is isomorphic to a moduli space of Bridgeland-stable objects in the derived category of $S$.  Thus Theorem \ref{theorem**} answers \cite[Question~1.2]{ms} in the untwisted case.

Hassett \cite[Prop.~6.1.3]{hassett} showed that if the generic $X \in \mathcal C_d$ has $F$ isomorphic to $\Hilb^2(S)$ for some K3 surface $S$ then
\begin{equation} \label{condition***} \tag{$*{*}*$}
d\text{ is of the form }\frac{2n^2+2n+2}{a^2}\text{ for some }n,a \in \Z,
\end{equation}
and proved a partial converse \cite[Thm.~6.1.4]{hassett}.  Thanks to the global Torelli theorem for hyperk\"ahler manifolds \cite{verbitsky,huybrechts_torelli,survey_of_torelli} we can now prove a more complete result:
\begin{thm} \label{theorem***}
The following are equivalent:
\begin{enumerate}
\item $X \in \mathcal C_d$ for some $d$ satisfying \eqref{condition***}.
\item $F$ is birational to $\Hilb^2(S)$ for some K3 surface $S$.
\end{enumerate}
\end{thm}

\begin{table}[b]
\caption{Comparison of numerical conditions.} \label{200}
\begin{center}
\begin{tabular}{ccccc}
\begin{tabular}{c|c|c}
$d$ & \eqref{condition**} & \eqref{condition***} \\
\hline
8 & & \\
12 & & \\
14 & x & x \\
18 & & \\
20 & & \\
24 & & \\
26 & x & x \\
30 & & \\
32 & & \\
36 & & \\
38 & x & x \\ % n = 30, a = 7
42 & x & x \\
44 & & \\
48 & & \\
50 & & \\
54 & & \\
56 & & \\
60 & & \\
62 & x & x \\
66 & & \\
68 & & \\
72 & & \\
\end{tabular}
& &
\begin{tabular}{c|c|c}
$d$ & \eqref{condition**} & \eqref{condition***} \\
\hline
74 & x & \\ % search up to a = 12
78 & x & \\ % search up to a = 4
80 & & \\
84 & & \\
86 & x & x \\
90 & & \\
92 & & \\
96 & & \\
98 & x & \\ % 2*98 = 14^2
102 & & \\
104 & & \\
108 & & \\
110 & & \\
114 & x & x \\
116 & & \\
120 & & \\
122 & x & x \\ % n = 2819, a = 361
126 & & \\
128 & & \\
132 & & \\
134 & x & x \\ % n = 4526, a = 553
138 & & \\
\end{tabular}
& &
\begin{tabular}{c|c|c}
$d$ & \eqref{condition**} & \eqref{condition***} \\
\hline
140 & & \\
144 & & \\
146 & x & x \\
150 & & \\
152 & & \\
156 & & \\
158 & x & \\ % search up to a = 1440
162 & & \\
164 & & \\
168 & & \\
170 & & \\
174 & & \\
176 & & \\
180 & & \\
182 & x & x \\
186 & x & x \\ % n = 66, a = 7
188 & & \\
192 & & \\
194 & x & x \\ % n = 423, a = 43
198 & & \\
200 & & \\
& & \\
\end{tabular}
\end{tabular}
\end{center}
\end{table}

In contrast to \eqref{condition**}, it is hard to tell at a glance whether a number $d$ satisfies \eqref{condition***}.  On the one hand \eqref{condition***} implies \eqref{condition**}, but it is strictly stronger: Hassett remarks in \cite[\S6.1]{hassett} that 74 satisfies \eqref{condition**} but not \eqref{condition***}.  To address the question systematically, observe that $d$ satisifies \eqref{condition***} if and only if there is an integral solution to the Pell-type equation $m^2 - 2da^2 = -3$; just substitute $m = 2n+1$.  If such an equation has any solution then it has one with $a$ below an explicit bound \cite[Thm.~4.2.7]{aac}.  It is then straightforward to have a computer search for solutions up to this bound.  Table \ref{200} lists all $d$ up to 200 that satisfy \eqref{condition*}, indicating whether they satisfy \eqref{condition**} and \eqref{condition***}.  I do not know any nice characterization of \eqref{condition***} in terms of the $A_2$ lattice.

\subsection*{Outline}
In \S\ref{markman-mukai} we review Markman's Mukai lattice for a variety $Y$ of \Kthreen-type, which governs the global Torelli theorem for such varieties.  We give criteria in terms of this lattice for $Y$ to be birational to a moduli space of sheaves or Hilbert scheme of $n$ points on a K3 surface.

In \S\ref{at-mukai} we review Kuznetsov's K3 category $\A$ associated to $X$, the special classes $\lambda_1, \lambda_2 \in \Knum(\A)$, and the Mukai lattice $\Ktop(\A)$ introduced in \cite{at}.  We prove that
\begin{equation} \label{ext_of_bd}
H^2(F,\Z)(1) \cong {\lambda_1}^\perp \subset \Ktop(\A).
\end{equation}
This extends Beauville and Donagi's result \cite[Prop.~6]{bd} that $H^2_\prim(F,\Z)(1) \cong H^4_\prim(X,\Z)(2)$, since the latter is Hodge-isometric to $\langle \lambda_1, \lambda_2 \rangle^\perp \subset \Ktop(\A)$.  From \eqref{ext_of_bd} we deduce that $\Ktop(\A)(-1)$ is the Markman--Mukai lattice of $F$.  All this is consistent with Kuznetsov and Markushevich's result \cite[\S5]{km} that $F$ is a moduli space of objects in the numerical class $\lambda_1 \in \Knum(\A)$.  

With this lattice theory in hand, we prove Theorems \ref{theorem**} and \ref{theorem***} in \S\ref{proof_of_theorems}.

\subsection*{Convention}  Since we are speaking about transcendental lattices and moduli spaces of sheaves, we will take all K3 surfaces to be projective unless otherwise stated.

\subsection*{Acknowledgements}  This paper grew out of a long-running conversation with Richard Thomas.  I also thank Fran\c{c}ois Charles and Giovanni Mongardi for helpful discussions, and Kevin Buzzard, Emanuele Macr\`i, and Eyal Markman for expert advice.  This work was partly done while visiting the Hausdorff Research Institute for Mathematics in Bonn, Germany, and was partly supported by NSF grant no.\ DMS--0905923.

\section{The Markman--Mukai lattice of a variety of \texorpdfstring{\Kthreen}{K3\textasciicircum[n]}-type} \label{markman-mukai}
A \emph{variety of \Kthreen-type} is a smooth projective variety $Y$ deformation-equivalent to the Hilbert scheme of $n$ points of a K3 surface, $n \ge 2$.  The second cohomology group $H^2(Y,\Z)$ carries a quadratic form $q$, the \emph{Beauville--Bogomolov--Fujiki form}, under which it is a lattice of discriminant $-2n+2$ and signature $(3,20)$.  Markman \cite[\S9]{survey_of_torelli} has described an extension of lattices and weight-2 Hodge structures $H^2(Y,\Z) \subset \tilde\Lambda$ with the following properties:
\pagebreak
\begin{thm}[Markman\footnote{This summary is borrowed from \cite[\S1]{bht}.}] \label{markman_thm} \ 
\begin{enumerate}
\item As a lattice, $\tilde\Lambda$ is isomorphic to $U^4 \oplus (-E_8)^2$.
\item The orthogonal $H^2(Y,\Z)^\perp \subset \tilde\Lambda$ is generated by a primitive vector of square $2n-2$.
\item If $Y$ is a moduli space of sheaves on a K3 surface $S$ with Mukai vector $v \in H^*(S,\Z)$ then the extension $H^2(Y,\Z) \subset \tilde\Lambda$ is naturally identified with $v^\perp \subset H^*(S,\Z)$.
\item $Y_1$ and $Y_2$ are birational if and only if there is a Hodge isometry $\tilde\Lambda_1 \to \tilde\Lambda_2$ taking $H^2(Y_1,\Z)$ isomorphically to $H^2(Y_2,\Z)$.
\end{enumerate}
\end{thm}
\noindent Let $\tilde\Lambda_\alg \supset H^{1,1}(Y,\Z)$ denote the algebraic part of $\tilde\Lambda$, that is, the integral classes of type $(1,1)$.

\begin{prop} \label{moduli_space_criterion}
Let $Y$ be a variety of \Kthreen-type, $n \ge 2$.  Then the following are equivalent:\footnote{Mongardi and Wandel have proved a similar result independently in \cite[Prop.~2.3]{mw}.}
\begin{enumerate}
\item $\tilde\Lambda_\alg$ contains a copy of the hyperbolic plane $U = \left( \begin{smallmatrix} 0 & 1 \\ 1 & 0 \end{smallmatrix} \right)$.
\item The transcendental lattice $T_Y \subset H^2(Y,\Z)$ is Hodge-isometric to $T_S$ for some K3 surface $S$.
\item $Y$ is birational to a moduli space of stable sheaves on $S$.
\end{enumerate}
\end{prop}

\begin{proof}
(c) $\Rightarrow$ (a): This is immediate from Theorem \ref{markman_thm}, since the algebraic part of $H^*(S,\Z)$ contains a copy of $U$ spanned by $H^0$ and $H^4$.

\smallskip
(a) $\Rightarrow$ (b): Let $L = U^\perp \subset \tilde\Lambda$.  As a lattice this is isomorphic to $U^3 \oplus (-E_8)^2$, so by the global Torelli theorem it is Hodge-isometric to $H^2(S,\Z)$ for some analytic K3 surface $S$.  In fact $S$ is projective, as follows.  By Huybrechts' projectivity criterion \cite[Thm.~3.11]{huybrechts_basic} there is a $c \in H^{1,1}(Y,\Z)$ with $q(c) > 0$.  Let $v$ be a primitive generator of $H^2(Y,\Z)^\perp \subset \tilde\Lambda$; then $q(v) = 2n-2 > 0$.  Thus $c$ and $v$ span a positive definite sublattice of $\tilde\Lambda$.  This cannot be contained in $U$, which is indefinite, so $\langle c, v \rangle \cap L$ contains a class of positive square, so $S$ is projective by Huybrechts' criterion.

Now $T_S$ is the transcendental part of $L$, which is the transcendental part of $\tilde\Lambda$, which is $T_Y$.

\smallskip
(b) $\Rightarrow$ (c):  We have a Hodge isometry $\phi\colon T_Y \to T_S$, and primitive embeddings $T_Y \subset \tilde\Lambda \cong U^4 \oplus (-E_8)^2$ and $T_S \subset H^*(S,\Z) \cong U^4 \oplus (-E_8)^2$.  The orthgonal ${T_S}^\perp$ contains a copy of $U$, so by \cite[Prop.~3.8]{orlov} any two primitive embeddings $T_S \hookrightarrow U^4 \oplus (-E_8)^2$ differ by an automorphism of $U^4 \oplus (-E_8)^2$.  Thus the lattice isomorphism $\phi\colon T_Y \to T_S$ extends to a lattice isomorphism $\tilde\phi\colon \tilde\Lambda \to H^*(S,\Z)$.  Since $\phi$ is a Hodge isometry, it takes $H^{2,0}(Y)$ to $H^{2,0}(S)$, so the extension $\tilde\phi$ does as well, so $\tilde\phi$ is a Hodge isometry.

Again let $v \in \tilde\Lambda$ be a primitive generator of $H^2(Y,\Z)^\perp \subset \tilde\Lambda$, and write $\tilde\phi(v) = (r,c,s) \in H^*(S,\Z)$.  I claim that either $r > 0$, or we can modify $v$ and $\tilde\phi$ to make it so.  If $r < 0$, replace $v$ with $-v$.  If $r = 0$ and $s \ne 0$, compose $\tilde\phi$ with the Mukai reflection through $(1,0,1) \in H^*(S,\Z)$, so now $\tilde\phi(v) = (-s,c,0)$ and we are reduced to the previous case.  If $r = s = 0$, note that $c^2 = q(v) = 2n-2 > 0$, and compose $\tilde\phi$ with multiplication by $\exp(c) = (1,c,n-1)$, so now $\tilde\phi(v) = (0,c,n-1)$ and we are reduced to the previous case.

Now $\tilde\phi(v)$ is a Mukai vector of positive rank, so for a generic polarization of $S$ the moduli space $M$ of stable sheaves on $S$ with Mukai vector $\tilde\phi(v)$ is smooth and non-empty \cite{mukai}.  By construction $\tilde\phi$ is a Hodge isometry from $\tilde\Lambda$ to $H^*(S,\Z)$ taking $H^2(Y,\Z)$ isomorphically to $\tilde\phi(v)^\perp$, so $Y$ is birational to $M$ by Theorem \ref{markman_thm}.
\end{proof}

\begin{prop} \label{hilbert_scheme_criterion}
Let $Y$ be a variety of \Kthreen-type, $n \ge 2$, and let $v$ be a primitive generator of $H^2(Y,\Z)^\perp \subset \tilde\Lambda$.  Then the following are equivalent:
\begin{enumerate}
\item There is a vector $w \in \tilde\Lambda_\alg$ such that $v.w=-1$ and $w^2=0$.
\item $Y$ is birational to $\Hilb^n(S)$ for some K3 surface $S$.
\end{enumerate}
\end{prop}

\begin{proof}
(b) $\Rightarrow$ (a): This is immediate from Theorem \ref{markman_thm}, since $\Hilb^n(S)$ is the moduli space of sheaves with Mukai vector $v = (1,0,1-n) \in H^*(S,\Z)$; take $w = (0,0,1)$.

\smallskip
(a) $\Rightarrow$ (b): Observe that $e := v+(n-1)w$ and $f := -w$ satisfy $e^2 = f^2 = 0$ and $e.f = 1$, so they span a copy of $U$ in $\tilde\Lambda_\alg$.  Let $L = U^\perp = \langle v,w \rangle^\perp \subset \tilde\Lambda$.  As in the proof of Proposition \ref{moduli_space_criterion}, there is a projective K3 surface $S$ such that $H^2(S,\Z) \cong L$.  Thus we can produce a Hodge isometry from $\tilde\Lambda = U \oplus L$ to $H^*(S,\Z)$ that takes $v$ to $(1,0,1-n)$, so $Y$ is birational to $\Hilb^n(S)$ by Theorem \ref{markman_thm}.
\end{proof}

\section{The Markman--Mukai lattice of \texorpdfstring{$F$}{F}} \label{at-mukai}

Recall that $X$ is a smooth cubic fourfold and $F$ is the variety of lines on $X$.  Kuznetsov has observed that the triangulated category
\begin{align*}
\A &:= \langle \O_X, \O_X(1), \O_X(2) \rangle^\perp \subset D^b(\Coh(X)) \\
&:= \{ E \in D^b(\Coh(X)) : \Ext^*(\O_X(i),E)=0\ \mathrm{for}\ i=0,1,2 \}
\end{align*}
is like the derived category of a K3 surface in that it has the same Serre functor and Hochschild homology and cohomology, and has conjectured that $X$ is rational if and only if $\A$ is equivalent to the derived category of an actual K3 surface \cite{kuznetsov}.  By \cite{at}, this is essentially equivalent to having $X \in \mathcal C_d$ for some $d$ satisfying \eqref{condition**}.

Let $\Knum(\A)$ be the numerical Grothendieck group of $\A$, that is, $K(\A)$ modulo the kernel of the Euler pairing.  Let $\lambda_1, \lambda_2 \in \Knum(\A)$ be the classes of the projections of $\O_L(1)$ and $\O_L(2)$ into $\A$, where $L$ is any line on $X$.  The Euler pairing on the sublattice $\langle \lambda_1, \lambda_2 \rangle$ is $-A_2 = \left( \begin{smallmatrix} -2 & 1 \\ 1 & -2 \end{smallmatrix} \right)$.

A Mukai lattice for $\A$ was introduced in \cite[Def.~2.2]{at}:
\[ \Ktop(\A) := \{ \kappa \in \Ktop(X) : \chi([\O_X(i)], \kappa) = 0\ \mathrm{for}\ i=0,1,2 \}. \]
Here $\Ktop(X)$ is the Grothendieck group of topological vector bundles and $\chi$ is the Euler pairing, which is integer-valued and extends the Euler pairing on $\Knum(X)$.  It has a Hodge structure of K3 type pulled back via the Chern character or the Mukai vector
\[ \Ktop(\A) \otimes \mathbb C \hookrightarrow \bigoplus H^{2i}(X,\mathbb C)(i). \]
In \cite{at} this was called a weight-two Hodge structure, but it should really be called weight-zero.  We will need the following properties:
\begin{thm}[Addington, Thomas {\cite[\S\S2.3--2.4]{at}}] \label{at-mukai-properties} \ 
\begin{enumerate}
\item As a lattice, $\Ktop(\A)$ is isomorphic to $U^4 \oplus {E_8}^2$.  
\item The algebraic part of $\Ktop(\A)$ is isomorphic to $\Knum(\A)$.
\item $\langle \lambda_1, \lambda_2 \rangle^\perp \subset \Ktop(\A)$ is Hodge-isometric to $H^4_\prim(X,\Z)(2)$.
\item $X \in \mathcal C_d$ if and only if there is a primitive sublattice $M \subset \Knum(\A)$ of rank 3 and discriminant $d$ that contains $\lambda_1$ and $\lambda_2$.
\end{enumerate}
\end{thm}

\begin{prop}
Let $P \subset F \times X$ be the universal line and $p\colon P \to F$ and $q\colon P \to X$ the two projections.  Then the map $\phi$ from ${\lambda_1}^\perp \subset \Ktop(\A)$ to $H^2(F,\Z)(1)$ defined by $\phi(\kappa) = c_1(p_* q^* \kappa)$ is a Hodge isometry.
\end{prop}
\begin{proof}
Both ${\lambda_1}^\perp$ and $H^2(F,\Z)(1)$ are lattices of rank 23 and discriminant 2.  It is enough to show that $\phi$ is a Hodge isometry when tensored with $\Q$; a priori this only implies that $\phi$ embeds ${\lambda_1}^\perp$ as a finite-index sublattice of $H^2(F,\Z)(1)$, but since they have the same discriminant the index must in fact be 1.

By the Riemann--Roch formula \cite[\S3]{ah}, $\phi(\kappa)$ is the degree-2 part of
\begin{equation} \label{rr}
p_*(q^*(\ch(\kappa)) \cup \td(T_p)),
\end{equation}
where $T_p$ is the relative tangent bundle of the $\P^1$-bundle $p\colon P \to F$.  First we calculate $\td(T_p)$.  Let $h \in H^2(X,\Z)$ be the hyperplane class.  Let $S$ be the restriction to $F$ of the tautological sub-bundle on $\Gr(2,6)$.  Then $g := -c_1(S) \in H^2(F,\Z)$ is the hyperplane class in the Pl\"ucker embedding.  The universal line $P$ is the projectivization $\P S$, and $\O_{\P S}(1) = q^* \O_X(1)$.  Since $T_p$ is line bundle, we can take determinants in the Euler sequence
\[ 0 \to \O_{\P S} \to \O_{\P S}(1) \otimes p^* S \to T_p \to 0 \]
to get $T_p = \O_{\P S}(2) \otimes p^* \det S$.  Thus
\begin{equation} \label{todd}
\td(T_p) = 1 + \tfrac12 (2q^* h - p^* g) + \tfrac1{12} (2q^* h - p^* g)^2 + \dotsb.
\end{equation}

The orthogonal to $\lambda_1$ in $\langle \lambda_1, \lambda_2 \rangle$ is generated by $\lambda_1 + 2\lambda_2$.  Since we are tensoring with $\Q$, we have orthogonal direct sums
\begin{gather}
{\lambda_1}^\perp = \langle \lambda_1 + 2\lambda_2 \rangle \oplus \langle \lambda_1, \lambda_2 \rangle^\perp \label{lambda_perp_splitting} \\
H^2(F,\Q) = \langle g \rangle \oplus H^2_\prim(F,\Q). \label{H2F_splitting}
\end{gather}
By \cite[Prop.~2.3]{at}, the Chern character\footnote{In fact \cite[Prop.~2.3]{at} says that the Mukai vector gives such a Hodge isometry, but since $\td(X)$ is a polynomial in $h$, multiplying by $\sqrt{\td(X)}$ does not affect $H^4_\prim(X,\Q)$.} gives a Hodge isometry from the second summand of \eqref{lambda_perp_splitting} to $H^4_\prim(X,\Q)(2)$.  By \cite[Prop.~6]{bd}, $p_* q^*$ gives a Hodge isometry from this to the second summand of \eqref{H2F_splitting}.  Since the degree-0 part of $\td(T_p)$ is $1$, we see that for $\alpha \in H^4(X,\Q)$, the degree-2 part of $p_*(q^* \alpha \cup \td(T_p))$ is just $p_* q^* \alpha$.  Thus $\phi$ gives a Hodge isometry from the second summand of \eqref{lambda_perp_splitting} to the second summand of \eqref{H2F_splitting}.

For the first summands of \eqref{lambda_perp_splitting} and \eqref{H2F_splitting}, observe that the Euler square of $\lambda_1 + 2\lambda_2$ is $-6$, and by \cite[\S2.1]{hassett} we have $q(g) = -6$ as well (the minus sign comes because we have twisted down to weight zero).  Thus it is enough to show that
\begin{equation} \label{desired}
\phi(\lambda_1 + 2\lambda_2) = g.
\end{equation}
To calculate $\ch(\lambda_1 + 2 \lambda_2)$, recall that $\lambda_i$ is the class of the left mutation of $\O_L(i)$ past $\O_X(2)$, $\O_X(1)$, and $\O_X$, where $L$ is any line on $X$, so a straightforward calculation gives
\begin{align*}
\lambda_1 &= [\O_L(1)] - [\O_X(1)] + 4[\O_X] \\
\lambda_2 &= [\O_L(2)] - [\O_X(2)] + 4[\O_X(1)] - 6[\O_X]
\end{align*}
and thus
\[ \ch(\lambda_1 + 2 \lambda_2) = -3 + 3h - \tfrac12 h^2 + \dotsb. \]
By \cite[\S2.1]{hassett} we have $p_* q^* h^2 = g$.  We also have $p_* q^* h = 1$: to see this, take a smooth hyperplane section $X \cap H$ and take its preimage under $q$; this is the blow-up of $F$ along the surface of lines contained in the cubic threefold $X \cap H$, hence is generically 1-to-1 over $F$.  Combining these facts with \eqref{rr} and \eqref{todd} we get \eqref{desired}.
\end{proof}

\begin{cor} \label{the_corollary}
The embedding $H^2(F,\Z) \subset \Ktop(\A)(-1)$ given by the previous proposition can be identified with Markman's embedding $H^2(F,\Z) \subset \tilde\Lambda$ discussed in \S\ref{markman-mukai}.
\end{cor}
\begin{proof}
If $n=2$ or if $n-1$ is a prime power then for any $Y$ of \Kthreen-type, any two primitive embeddings of $H^2(Y,\Z)$ into $U^4 \oplus (-E_8)^2$ differ by an automorphism of the latter \cite[\S4.1]{integral_constraints}.
\end{proof}

\section{Proofs of Theorems 1 and 2} \label{proof_of_theorems}

{\renewcommand{\thethm}{\ref{theorem**}}
\addtocounter{thm}{-1}
\begin{thm}
The following are equivalent:
\begin{enumerate}
\item $X \in \mathcal C_d$ for some $d$ satisfying \eqref{condition**}.
\item The transcendental lattice $T_X \subset H^4(X,\Z)$ is Hodge-isometric to $T_S(-1)$ for some K3 surface $S$.
\item $F$ is birational to a moduli space of stable sheaves on $S$.
\end{enumerate}
\end{thm}}
\begin{proof}
By \cite[Thm.~3.1]{at}, condition (a) holds if and only if $\Knum(\A)$ contains a copy of $U \cong -U$.  Moreover we have $T_X \cong T_F(-1)$.  Thus the theorem follows from Corollary \ref{the_corollary} and Proposition \ref{moduli_space_criterion}.
\end{proof}

To prove Theorem \ref{theorem***} we will have to work in a basis:
\begin{lem} \label{basis_lemma}
If $X \in \mathcal C_d$ then there is a $\tau \in \Knum(\A)$ such that $\langle \lambda_1, \lambda_2, \tau \rangle$ is a primitive sublattice of discriminant $d$ with Euler pairing
\[ \begin{pmatrix}
-2 & 1 & 0 \\
1 & -2 & 0 \\
0 & 0 & 2k
\end{pmatrix}
 \qquad\mathrm{when\ }d = 6k,\ \mathrm{or} \]
\[ \begin{pmatrix}
-2 & 1 & 0 \\
1 & -2 & 1 \\
0 & 1 & 2k
\end{pmatrix}
 \qquad\mathrm{when\ }d = 6k+2. \]
\end{lem}
\begin{proof}
By Theorem \ref{at-mukai-properties}(d), we can choose a $\tau \in \Knum(\A)$ such that $\langle \lambda_1, \lambda_2, \tau \rangle$ is a primitive sublattice of discriminant $d$.  Write the Euler pairing as
\[ \begin{pmatrix}
-2 & 1 & a \\
1 & -2 & * \\
a & * & *
\end{pmatrix} \]
for some $a \in \Z$.  Replace $\tau$ with $\tau - a\lambda_2$; then the Euler pairing becomes
\[ \begin{pmatrix}
-2 & 1 & 0 \\
1 & -2 & 3b+c \\
0 & 3b+c & *
\end{pmatrix} \]
for some $b$ and some $-1 \le c \le 1$.  Replace $\tau$ with $\tau + b(\lambda_1+2\lambda_2)$; then the Euler pairing becomes
\[ \begin{pmatrix}
-2 & 1 & 0 \\
1 & -2 & c \\
0 & c & 2k
\end{pmatrix} \]
for some $k$, since $\Knum(\A)$ is an even lattice.  If $c=0$ this has determinant $6k$.  If $c=1$ this has determinant $6k+2$.    If $c=-1$, replace $\tau$ with $-\tau$ to get back to the previous case.
\end{proof}

{\renewcommand{\thethm}{\ref{theorem***}}
\addtocounter{thm}{-1}
\begin{thm}
The following are equivalent:
\begin{enumerate}
\item $X \in \mathcal C_d$ for some $d$ satisfying \eqref{condition***}.
\item $F$ is birational to $\Hilb^2(S)$ for some K3 surface $S$.
\end{enumerate}
\end{thm}}
\begin{proof}
We will show that condition (a) holds if and only if there is a $w \in \Knum(\A)$ such that
\begin{equation} \label{kappa}
\chi(\lambda_1, w) = 1 \mathrm{\ and\ } \chi(w,w) = 0.
\end{equation}
Then the theorem follows from Corollary \ref{the_corollary} and Proposition \ref{hilbert_scheme_criterion}.  

If there is such a $w$, let $L = \langle \lambda_1, \lambda_2, w \rangle \subset \Knum(\A)$.  By hypothesis, the Euler pairing on $L$ is
\[ \begin{pmatrix}
-2 & 1 & 1 \\
1 & -2 & n \\
1 & n & 0
\end{pmatrix} \]
for some $n \in \Z$, so $\disc(L) = 2n^2 + 2n + 2$.  Let $M$ be the saturation of $L$, let $a$ be the index of $L$ in $M$, and let $d = \disc(M)$.  Then $\disc(L) = a^2 d$, and $X \in \mathcal C_d$ by Theorem \ref{at-mukai-properties}(d).

Conversely, suppose $X \in \mathcal C_d$ for some $d$ satisfying \eqref{condition***}.  Choose integers $n$ and $a$ such that
\[ d a^2 = 2n^2 + 2n + 2. \]
Recall that $d$ is even.  Since $2n^2+2n+2$ satisfies \eqref{condition**} we see that $a$ is a product of primes $p \equiv 1 \pmod 3$, and in particular $a \equiv 1 \pmod 3$.  We consider three cases.

\smallskip
Case 1: $n \equiv 1 \pmod 3$.  In this case we find that $d \equiv 0 \pmod 6$.  Write $d = 6k$.  By Lemma \ref{basis_lemma} there is a $\tau \in \Knum(\A)$ such that the Euler pairing on $\langle \lambda_1, \lambda_2, \tau \rangle$ is
\[ \begin{pmatrix}
-2 & 1 & 0 \\
1 & -2 & 0 \\
0 & 0 & 2k
\end{pmatrix}. \]
Let $m = (n-1)/3$, which is an integer; then we find that
\[ w := m\lambda_1 + (2m+1)\lambda_2 + a \tau \]
satisfies \eqref{kappa}.

\smallskip
Case 2: $n \equiv 2 \pmod 3$.  In this case we find that $d \equiv 2 \pmod 6$.  Write $d = 6k+2$.  By Lemma \ref{basis_lemma} there is a $\tau \in \Knum(\A)$ such that the Euler pairing on $\langle \lambda_1, \lambda_2, \tau \rangle$ is
\[ \begin{pmatrix}
-2 & 1 & 0 \\
1 & -2 & 1 \\
0 & 1 & 2k
\end{pmatrix}. \]
Let $m = (a-n-2)/3$, which is an integer; then we find that
\[ w := m\lambda_1 + (2m+1)\lambda_2 + a \tau \]
satisfies \eqref{kappa}.

\smallskip
Case 3: $n \equiv 0 \pmod 3$.  Again we find that $d \equiv 2 \pmod 6$.  Argue as in the previous case but with $m = (a+n-1)/3$.
\end{proof}

\bibliographystyle{plain}
\bibliography{hassett_vs_galkin}

\end{document}